\documentclass{my_aims}

\usepackage{amsmath}
\usepackage{paralist}

\usepackage{cite}

\usepackage[colorlinks=true]{hyperref}
\hypersetup{urlcolor=blue, citecolor=red}

\def\ppages{X--XX}

\newtheorem{theorem}{Theorem}[section]
\newtheorem{corollary}{Corollary}

\newtheorem{lemma}[theorem]{Lemma}
\newtheorem{proposition}{Proposition}

\theoremstyle{definition}
\newtheorem{definition}[theorem]{Definition}
\newtheorem{remark}{Remark}
\newtheorem{ex}{Example}

\renewcommand{\subjclassname}{\textup{2010} Mathematics Subject Classification}

\title[Hahn's symmetric variational calculus]{Hahn's Symmetric Quantum Variational Calculus}

\author[A. M. C. Brito da Cruz, N. Martins and D. F. M. Torres]{}

\subjclass{39A13, 49K05.}

\keywords{Hahn's symmetric calculus, Euler--Lagrange difference equations,
calculus of variations, Leitmann's principle, quantum calculus.}

\email{artur.cruz@estsetubal.ips.pt}
\email{natalia@ua.pt}
\email{delfim@ua.pt}

\thanks{Part of first author's Ph.D., which is carried out
at the University of Aveiro under the \emph{Doctoral Programme
in Mathematics and Applications} of Universities of Aveiro and Minho.}

\begin{document}

\maketitle


\centerline{\scshape Artur M. C. Brito da Cruz}
\medskip
{\footnotesize
\centerline{Escola Superior de Tecnologia de Set\'{u}bal,
Estefanilha,
2910-761 Set\'{u}bal,
Portugal}
}
\medskip
{\footnotesize
\centerline{Center for Research and Development in Mathematics and Applications}
\centerline{Department of Mathematics, University of Aveiro, 3810-193 Aveiro, Portugal}
}

\medskip


\centerline{\scshape Nat\'{a}lia Martins and Delfim F. M. Torres}
\medskip
{\footnotesize
\centerline{Center for Research and Development in Mathematics and Applications}
\centerline{Department of Mathematics, University of Aveiro, 3810-193 Aveiro, Portugal}
}


\bigskip

\centerline{To George Leitmann on the occasion of his 87th birthday}


\begin{abstract}
We introduce and develop the Hahn symmetric quantum calculus with applications
to the calculus of variations. Namely, we obtain a necessary optimality condition
of Euler--Lagrange type and a sufficient optimality condition for variational
problems within the context of Hahn's symmetric calculus. Moreover, we show the
effectiveness of Leitmann's direct method when applied to Hahn's symmetric
variational calculus. Illustrative examples are provided.
\end{abstract}


\section{Introduction}

Due to its many applications, quantum operators are recently subject
to an increase number of investigations \cite{Malinowska,MyID:146,MyID:220}.
The use of quantum differential operators, instead of classical derivatives,
is useful because they allow to deal with sets of nondifferentiable functions
\cite{MyID:188,MyID:111}. Applications include several fields of physics,
such as cosmic strings and black holes \cite{Strominger}, quantum mechanics
\cite{Feynman,Youm}, nuclear and high energy physics \cite{Lavagno},
just to mention a few. In particular, the $q$-symmetric quantum calculus
has applications in quantum mechanics \cite{Lavagno2}.

In 1949, Hahn introduced his quantum difference operator \cite{Hahn},
which is a generalization of the quantum $q$-difference operator defined
by Jackson \cite{Jackson}. However, only in 2009, Aldwoah \cite{Aldwoah}
defined the inverse of Hahn's difference operator, and short after,
Malinowska and Torres \cite{Malinowska} introduced and investigated
the Hahn quantum variational calculus. For a deep understanding
of quantum calculus, we refer the reader to
\cite{Aldwoah2,Boole,Cruz,Ernst,Kac,Koekoek} and references therein.

For a fixed $q\in \left] 0,1\right[$ and an $\omega \geq 0$, we introduce
here the Hahn symmetric difference operator of function $f$ at point
$t\neq \displaystyle \frac{\omega }{1-q}$ by
\begin{equation*}
\tilde{D}_{q,\omega }\left[ y\right] \left( t\right)
=\frac{f\left( qt+\omega \right) - f\left( q^{-1}\left( t
-\omega \right) \right) }{\left( q-q^{-1}\right)
t+\left( 1+q^{-1}\right) \omega}.
\end{equation*}
Our main aim is to establish a necessary optimality condition
and a sufficient optimality condition
for the Hahn symmetric variational problem
\begin{equation}
\label{P} \tag{P}
\begin{gathered}
\mathcal{L}(y) =\int_{a}^{b}L\left( t,y^{\sigma }\left( t\right),
\tilde{D}_{q,\omega }\left[ y\right] \left( t\right) \right) \tilde{d} _{q,\omega }t
\longrightarrow \textrm{extremize} \\
y\in \mathcal{Y}^{1}\left( \left[ a,b\right]_{q,\omega} ,\mathbb{R} \right) \\
y\left( a\right) =\alpha, \quad y\left(b\right) =\beta ,
\end{gathered}
\end{equation}
where $\alpha $ and $\beta $ are fixed real numbers, and extremize means maximize or minimize.
Problem \eqref{P} will be clear and precise after definitions of Section~\ref{Pre}.
We assume that the Lagrangian $L$ satisfies the following hypotheses:
\begin{enumerate}
\item[(H1)] $\left( u,v\right) \rightarrow L\left( t,u,v\right) $
is a $C^{1}\left(\mathbb{R}^{2}, \mathbb{R}\right)$ function for any $t\in I$;

\item[(H2)] $t\rightarrow L\left( t,y^{\sigma }\left( t\right),
\tilde{D}_{q,\omega }\left[ y\right]\left(t\right) \right)$
is continuous at $\omega_{0}$ for any admissible function $y$;

\item[(H3)] functions $t\rightarrow \partial_{i+2}
L\left(t, y^{\sigma}\left( t\right),
\tilde{D}_{q,\omega }\left[ y\right]\left(t\right) \right)$ belong to
$\mathcal{Y}^{1}\left( \left[a,b\right]_{q,\omega}, \mathbb{R}\right)$
for all admissible $y$, $i=0,1$;
\end{enumerate}
where $I$ is an interval of $\mathbb{R}$ containing
$\omega _{0}:=\displaystyle\frac{\omega }{1-q}$,
$a,b\in I$, $a<b$, and $\partial_{j}L$ denotes the partial
derivative of $L$ with respect to its $j$th argument.

In Section~\ref{Pre} we introduce the necessary definitions and prove
some basic results for the Hahn symmetric calculus.
In Section~\ref{M} we formulate and prove our main results for the
Hahn symmetric variational calculus. New results include a necessary
optimality condition (Theorem~\ref{Euler}) and a sufficient
optimality condition (Theorem~\ref{Suficiente}) to problem \eqref{P}.
In Section~\ref{L} we show that Leitmann's direct method can also be
applied to variational problems within Hahn's symmetric variational calculus.
Leitmann introduced his direct method in the sixties of the 20th century \cite{Leitmann3},
and the approach has recently proven to be universal: see, \textrm{e.g.},
\cite{MyID:154,Leitmann1,Leitmann2,Leitmann4,Leitmann5,Leitmann6,Malinowska2,delfim:GL}.


\section{Hahn's symmetric calculus}
\label{Pre}

Let $q\in \left] 0,1\right[ $ and $\omega \geq 0$ be real fixed numbers.
Throughout the text, we make the assumption that $I$ is an interval
(bounded or unbounded) of $\mathbb{R}$ containing
$\omega _{0}:=\displaystyle\frac{\omega }{1-q}$.
We denote by $I^{q,\omega }$ the set
$I^{q,\omega }:=qI+\omega :=\left\{qt+\omega: t\in I\right\}$.
Note that $I^{q,\omega }\subseteq I$ and, for all $t\in I^{q,\omega }$,
one has $q^{-1}\left(t-\omega\right)\in I$ . For $k\in \mathbb{N}_{0}$,
$$
\left[ k\right] _{q}:=\displaystyle\frac{1-q^{k}}{1-q}.
$$

\begin{definition}
\label{def:dhsd}
Let $f$ be a real function defined on $I$. The Hahn symmetric difference operator
of $f$ at a point $t\in I^{q,\omega }\backslash \left\{ \omega_{0}\right\}$
is defined by
\begin{equation*}
\tilde{D}_{q,\omega }\left[ f\right] \left( t\right)
=\frac{f\left( qt+\omega \right)
-f\left( q^{-1}\left( t-\omega \right) \right) }{\left( q-q^{-1}\right)
t+\left( 1+q^{-1}\right) \omega },
\end{equation*}
while $\tilde{D}_{q,\omega }\left[ f\right] \left( \omega _{0}\right)
:=f^{\prime }\left( \omega _{0}\right)$,
provided $f$ is differentiable at $\omega _{0}$ (in the classical sense).
We call to $\tilde{D}_{q,\omega }\left[ f\right]$
the Hahn symmetric derivative of $f$.
\end{definition}

\begin{remark}
If $\omega =0$, then the Hahn symmetric difference operator $\tilde{D}_{q,\omega }$
coincides with the $q$-symmetric difference operator $\tilde{D}_{q}$: if $t\neq 0$, then
\begin{equation*}
\tilde{D}_{q,0}\left[ f\right](t)
=\frac{f\left( qt\right) -f\left( q^{-1}t\right) }{\left( q-q^{-1}\right) t}
=: \tilde{D}_{q}\left[ f\right] \left( t\right);
\end{equation*}
for $t = 0$ and $f$ differentiable at $0$, $\tilde{D}_{q,0}\left[ f\right](0)
= f^{\prime }\left( 0\right) =: \tilde{D}_{q}\left[ f\right] \left( 0\right)$.
\end{remark}

\begin{remark}
If $\omega> 0$ and we let $q\rightarrow 1$ in Definition~\ref{def:dhsd},
then we obtain the well known symmetric difference operator $\tilde{D}_{\omega}$:
\begin{equation*}
\tilde{D}_{\omega}\left[ f\right] \left( t\right)
=\frac{f\left( t+\omega\right) -f\left( t-\omega\right) }{2\omega}.
\end{equation*}
\end{remark}

\begin{remark}
If $f$ is differentiable at $t\in I^{q,\omega }$ in the classical sense, then
\begin{equation*}
\lim_{(q,\omega)\rightarrow (1,0)}\tilde{D}_{q,\omega }\left[
f \right] \left( t\right) =f^{\prime }\left( t\right) .
\end{equation*}
\end{remark}

In what follows we make use of the operator $\sigma$ defined by
$\sigma \left( t\right) :=qt+\omega $, $t\in I$.
Note that the inverse operator of $\sigma$, $\sigma^{-1}$,
is defined by $\sigma^{-1}\left(t\right):=q^{-1}\left(t-\omega\right)$.
Moreover, Aldwoah \cite[Lemma~6.1.1]{Aldwoah} proved the following useful result.

\begin{lemma}[\cite{Aldwoah}]
\label{lemma:2.2}
Let $k\in \mathbb{N}$ and $t\in I$. Then,
\begin{enumerate}
\item $\sigma^{k}\left( t\right) =\underset{k\text{ times}}{\underbrace{\sigma
\circ \sigma \circ \cdots \circ \sigma }}\left( t\right)
=q^{k}t+\omega\left[k\right]_{q}$;

\item $\left( \sigma^{k}\left( t\right) \right)^{-1}
=\sigma^{-k}\left( t\right)
=\displaystyle q^{-k}\left( t-\omega \left[ k\right] _{q}\right)$.
\end{enumerate}
Furthermore, $\{\sigma ^{k}\left( t\right)\}_{k=1}^{\infty}$ is a decreasing
(resp. an increasing) sequence in $k$ when $t > \omega_0$ (resp. $t < \omega_0$) with
$$
\omega_0 = \inf_{k \in \mathbb{N}} \sigma^{k}(t)
\quad \left(resp. \ \ \omega_0 = \sup_{k \in \mathbb{N}} \sigma^{k}(t)\right).
$$
The sequence $\{\sigma^{-k}(t)\}_{k=1}^{\infty}$ is increasing
(resp. decreasing) when $t > \omega_0$ (resp. $t < \omega_0$) with
$$
\infty = \sup_{k \in \mathbb{N}} \sigma^{-k}(t)
\quad \left(resp. \ \ -\infty = \inf_{k \in \mathbb{N}} \sigma^{-k}(t)\right).
$$
\end{lemma}

For simplicity of notation, we write
$f\left( \sigma \left( t\right) \right) :=f^{\sigma}\left( t\right)$.

\begin{remark}
With above notations, if $t\in I^{q,\omega }\backslash \left\{ \omega _{0}\right\}$,
then the Hahn symmetric difference operator of $f$ at point $t$ can be written as
\begin{equation*}
\tilde{D}_{q,\omega }\left[ f\right] \left( t\right)
=\frac{f^{\sigma}\left( t\right) -f^{\sigma ^{-1}}\left(
t\right) }{\sigma \left( t\right)-\sigma ^{-1}\left(t\right)}.
\end{equation*}
\end{remark}

\begin{lemma}
\label{q^n} Let $n \in \mathbb{N}_{0}$ and $t\in I$. Then,
\begin{equation*}
\sigma^{n+1}\left( t\right) -\sigma ^{n-1}\left( t\right)
=q^{n}\left( \sigma\left( t\right) - \sigma ^{-1}\left( t\right) \right),
\end{equation*}
where $\sigma^{0}\equiv id$ is the identity function.
\end{lemma}

\begin{proof}
The equality follows by direct calculations:
\begin{equation*}
\begin{split}
\sigma^{n+1}(t) -\sigma ^{n-1}\left( t\right)&=q^{n+1}t
+\omega \left[ n+1\right] _{q}-q^{n-1}t
-\omega \left[ n-1\right]_{q}\\ &= q^{n}\left( q-q^{-1}\right)
t+\omega \left( q^{n}+q^{n-1}\right)\\ &=q^{n}\left( qt+\omega
-q^{-1}t+q^{-1}\omega \right)\\
&=q^{n}\left( \sigma \left( t\right) -\sigma ^{-1}\left( t\right)\right).
\end{split}
\end{equation*}
\end{proof}

The Hahn symmetric difference operator has the following properties.

\begin{theorem}
\label{props derivada}
Let $\alpha, \beta \in \mathbb{R}$ and $t\in I^{q,\omega }$.
If $f$ and $g$ are Hahn symmetric differentiable on $I$, then
\begin{enumerate}
\item $\tilde{D}_{q,\omega}\left[ \alpha f+\beta g\right] \left( t\right)
=\alpha\tilde{D}_{q,\omega}\left[ f\right] \left( t\right)
+\beta\tilde {D}_{q,\omega}\left[ g\right] \left( t\right)$;

\item $\tilde{D}_{q,\omega}\left[ fg\right] \left( t\right)
=\tilde {D}_{q,\omega}\left[ f\right] \left( t\right) g^{\sigma}\left( t\right)
+f^{\sigma^{-1}}\left(t\right) \tilde{D}_{q,\omega}\left[ g\right] \left(t\right)$;

\item $\tilde{D}_{q,\omega }\left[ \displaystyle\frac{f}{g}\right] \left(
t\right) =\displaystyle\frac{\tilde{D}_{q,\omega }\left[ f\right] \left( t\right)
g^{\sigma ^{-1}}\left( t\right) -f^{\sigma ^{-1}}\left( t\right)
\tilde{D}_{q,\omega }\left[ g\right] \left( t\right) }{g^{\sigma }\left( t\right)
g^{\sigma ^{-1}}\left( t\right) }$ if $g^{\sigma }\left( t\right)
g^{\sigma ^{-1}}\left( t\right) \neq 0$;

\item $\tilde{D}_{q,\omega }\left[ f\right] \equiv 0$
if, and only if, $f$ is constant on $I$.
\end{enumerate}
\end{theorem}

\begin{proof}
For $t=\omega_{0}$ the equalities are trivial (note that
$\sigma(\omega_{0})=\omega_{0}=\sigma^{-1}(\omega_{0})$).
We do the proof for $t\neq\omega_{0}$:
\begin{enumerate}
\item
\begin{align*}
\tilde{D}_{q,\omega}\left[ \alpha f+\beta g\right] \left( t\right)
&= \frac{\left( \alpha f+\beta g\right)^{\sigma}\left( t\right)
-\left(\alpha f +\beta g\right)^{\sigma^{-1}}\left( t\right)}{\sigma\left( t\right)
-\sigma^{-1}\left( t\right)}\\ &= \alpha\frac{f^{\sigma}\left( t\right)
-f^{\sigma^{-1}}\left(t\right)}{\sigma\left( t\right) -\sigma^{-1}\left( t\right) }
+\beta \frac{g^{\sigma}\left( t\right) -g^{\sigma^{-1}}\left( t\right)}{\sigma\left(t\right)
-\sigma^{-1}\left( t\right)}\\ &= \alpha\tilde{D}_{q,\omega}\left[ f\right] \left( t\right)
+\beta \tilde{D}_{q,\omega}\left[ g\right] \left( t\right) .
\end{align*}

\item
\begin{align*}
\tilde{D}_{q,\omega}\left[ fg\right] \left( t\right)
&=\frac{\left( fg\right) ^{\sigma}\left( t\right)
-\left( fg\right) ^{\sigma^{-1}}\left( t\right)}{\sigma\left( t\right) -\sigma^{-1}\left( t\right) }\\
&= \frac{f^{\sigma}\left( t\right) -f^{\sigma^{-1}}\left( t\right)}{\sigma \left( t\right)
-\sigma^{-1}\left( t\right) }g^{\sigma}\left( t\right)
+f^{\sigma^{-1}}\left( t\right) \frac{g^{\sigma}\left( t\right)
-g^{\sigma^{-1}}\left( t\right)}{\sigma\left( t\right)
-\sigma ^{-1}\left(t\right) } \\
&= \tilde{D}_{q,\omega}\left[ f\right]\left( t\right) g^{\sigma}\left( t\right)
+f^{\sigma^{-1}}\left( t\right) \tilde{D}_{q,\omega}\left[ g\right] \left( t\right).
\end{align*}

\item Because
\begin{align*}
\tilde{D}_{q,\omega }\left[ \frac{1}{g}\right] \left( t\right)
&=\frac{ \frac{1}{g^{\sigma }\left( t\right) }
-\frac{1}{g^{\sigma^{-1}}\left( t\right)}}{\sigma \left( t\right)
-\sigma ^{-1}\left( t\right) } \\
& =-\frac{1}{g^{\sigma }\left( t\right)
g^{\sigma ^{-1}}\left( t\right) } \frac{g^{\sigma }\left(
t\right) -g^{\sigma ^{-1}}\left( t\right) }{\sigma \left( t\right)
-\sigma ^{-1}\left( t\right) } \\
& =-\frac{\tilde{D}_{q,\omega }\left[ g\right]\left(
t\right)}{g^{\sigma}\left(t\right) g^{\sigma^{-1}}\left( t\right)},
\end{align*}
one has
\begin{align*}
\tilde{D}_{q,\omega }\left[ \frac{f}{g}\right] \left( t\right)
&= \tilde{D}_{q,\omega }\left[ f\frac{1}{g}\right] \left( t\right)\\
&= \tilde{D}_{q,\omega}\left[ f\right] \left( t\right) \frac{1}{g^{\sigma }\left( t\right) }
+f^{\sigma ^{-1}}\left( t\right) \tilde{D}_{q,\omega } \left[ \frac{1}{g}\right] \left(t\right)\\
& =\frac{\tilde{D}_{q,\omega }\left[ f\right] \left( t\right) }{g^{\sigma }\left( t\right) }
-f^{\sigma ^{-1}}\left( t\right)
\frac{\tilde{D}_{q,\omega }\left[ g\right] \left( t\right) }{g^{\sigma }\left( t\right)
g^{\sigma ^{-1}}\left( t\right) } \\ & =\frac{\tilde{D}_{q,\omega}\left[ f\right] \left( t\right)
g^{\sigma ^{-1}}\left( t\right) -f^{\sigma ^{-1}}\left( t\right)
\tilde{D}_{q,\omega } \left[ g\right] \left( t\right)}{g^{\sigma }\left(
t\right) g^{\sigma ^{-1}}\left( t\right) }.
\end{align*}

\item If $f$ is constant on $I$, then it is clear that
$\tilde{D}_{q,\omega }\left[f\right] \equiv 0$.
Suppose now that $\tilde{D}_{q,\omega }\left[f\right] \equiv 0$.
Then, for each $t\in I$, $\left( \tilde{D}_{q,\omega }\left[
f\right] \right)^{\sigma}(t) =0$ and, therefore, $f\left( t\right)
=f^{\sigma^{2}}\left( t\right)$. Hence, $f\left( t\right)
=f^{\sigma ^{2}}\left( t\right) =\cdots
=f^{\sigma^{2n}}\left( t\right)$ for each $n\in\mathbb{N}$ and
$t\in I$. Because $\lim_{n\rightarrow +\infty }f\left( t\right)
=\lim_{n\rightarrow +\infty}f^{\sigma ^{2n}}\left( t\right)$,
$\lim_{n\rightarrow + \infty}\sigma^{2n}\left(t\right)
=\omega_{0}$ (by Lemma~\ref{lemma:2.2}), and $f$ is continuous
at $\omega_{0}$, then $f\left( t\right) =f\left( \omega_{0}\right)$
for all $t \in I$.
\end{enumerate}
\end{proof}

\begin{lemma}
\label{composta}
For $t\in I$ one has $\tilde{D}_{q,\omega }\left[ f^{\sigma }\right]\left(t\right)
=q\tilde{D}_{q,\omega }\left[f\right] \left( \sigma \left(t\right) \right)$.
\end{lemma}

\begin{proof}
For each $t\in I\backslash \{\omega_{0}\}$ we have
\begin{equation*}
\tilde{D}_{q,\omega }\left[ f^{\sigma }\right] \left( t\right)
=\frac{ f^{\sigma ^{2}}\left( t\right)
-f\left( t\right) }{\sigma \left( t\right)
-\sigma^{-1}\left( t\right) }
\end{equation*}
and
\begin{eqnarray*}
\tilde{D}_{q,\omega }\left[ f\right] \left( \sigma \left( t\right) \right)
&=&\frac{f^{\sigma ^{2}}\left( t\right) -f\left( t\right) }{\sigma ^{2}\left( t\right)-t} \\
&=&\frac{f^{\sigma ^{2}}\left( t\right) -f\left( t\right) }{q\left( \sigma \left( t\right)
-\sigma ^{-1}\left( t\right) \right)} \ \ \text{ (see Lemma~\ref{q^n})}.
\end{eqnarray*}
We conclude that $\tilde{D}_{q,\omega }\left[ f^{\sigma }\right]\left( t\right)
=q\tilde{D}_{q,\omega }\left[ f\right] \left( \sigma \left( t\right) \right)$.
Finally, the intended result follows from the fact that
$\tilde{D}_{q,\omega }\left[ f^{\sigma }\right] \left( \omega_{0}\right)
=q\tilde{D}_{q,\omega }\left[f\right] \left( \omega_{0}\right)$.
\end{proof}

\begin{definition}
Let $a,b\in I$ and $a<b$. For $f:I\rightarrow \mathbb{R}$
the Hahn symmetric integral of $f$ from $a$ to $b$ is given by
\begin{equation*}
\int_{a}^{b}f\left( t\right) \tilde{d}_{q,\omega }t
=\int_{\omega_{0}}^{b}f\left( t\right) \tilde{d}_{q,\omega }t
-\int_{\omega_{0}}^{a}f\left( t\right)\tilde{d}_{q,\omega }t,
\end{equation*}
where
\begin{equation*}
\int_{\omega _{0}}^{x}f\left( t\right) \tilde{d}_{q,\omega }t
=\left( \sigma^{-1}\left( x\right)
-\sigma \left( x\right) \right) \sum_{n=0}^{+\infty}
q^{2n+1}f^{\sigma ^{2n+1}}\left( x\right), \quad x\in I\text{,}
\end{equation*}
provided the series converges at $x=a$ and $x=b$. In that case,
$f$ is said to be Hahn symmetric integrable on $[a,b]$.
We say that $f$ is Hahn symmetric integrable on $I$ if it is
Hahn symmetric integrable over $[a,b]$ for all $a,b\in I$.
\end{definition}

We now present two technical results that will be useful to prove
the fundamental theorem of Hahn's symmetric integral
calculus (Theorem~\ref{Fundamental}).

\begin{lemma}[cf. \cite{Aldwoah}]
\label{lema integral}
Let $a,b\in I$, $a<b$. If $f:I\rightarrow \mathbb{R}$ is continuous
at $\omega_{0}$, then, for $s\in \left[ a,b\right]$, the sequence
$\left( f^{\sigma^{2n+1}}\left( s\right) \right)_{n\in\mathbb{N}}$
converges uniformly to $f\left( \omega _{0}\right)$ on $I$.
\end{lemma}

The next result tell us that if a function $f$ is continuous
at $\omega_{0}$, then $f$ is Hahn's symmetric integrable.

\begin{corollary}[cf. \cite{Aldwoah}]
\label{corolario integral}
Let $a,b\in I$, $a<b$, and $f:I\rightarrow \mathbb{R}$ be continuous
at $\omega_{0}$. Then, for $s\in \left[ a,b\right]$, the series
$\sum_{n=0}^{+\infty }q^{2n+1}f^{\sigma ^{2n+1}}\left( s\right)$
is uniformly convergent on $I$.
\end{corollary}

\begin{theorem}[Fundamental theorem of the Hahn symmetric integral calculus]
\label{Fundamental}
Assume that $f:I\rightarrow \mathbb{R}$ is continuous
at $\omega _{0}$ and, for each $x\in I$, define
\begin{equation*}
F(x):=\int_{\omega _{0}}^{x}f\left( t\right) \tilde{d}_{q,\omega }t\text{.}
\end{equation*}
Then $F$ is continuous at $\omega_{0}$. Furthermore, $\tilde{D}_{q,\omega}[F](x)$
exists for every $x\in I^{q,\omega }$ with $\tilde{D}_{q,\omega}[F](x)=f(x)$.
Conversely,
\begin{equation*}
\int_{a}^{b}\tilde{D}_{q,\omega }\left[ f\right] \left( t\right)
\tilde{d}_{q,\omega }t=f\left( b\right) -f\left( a\right)
\end{equation*}
for all $a,b\in I$.
\end{theorem}

\begin{proof}
We note that function $F$ is continuous at $\omega _{0}$
by Corollary~\ref{corolario integral}. Let us begin by
considering $x\in I\backslash \{\omega _{0}\}$. Then,
\begin{align*}
\tilde{D}_{q,\omega }&\left[ \tau \mapsto \int_{0}^{\tau}
f\left( t\right) \tilde{d}_{q,\omega }t\right](x)\\
&= \frac{\int_{\omega _{0}}^{\sigma \left( x\right)}
f\left( t\right) \tilde{d}_{q,\omega }t
-\int_{\omega _{0}}^{\sigma^{-1}\left( x\right)}
f\left( t\right) \tilde{d}_{q,\omega }t}{\sigma \left(x\right)
-\sigma^{-1}\left( x\right) } \\
&= \frac{1}{\sigma \left( x\right) -\sigma ^{-1}\left( x\right) }\bigg\{\left[
\sigma ^{-1}\left( \sigma \left( x\right) \right)
-\sigma\left( \sigma \left( x\right) \right) \right] \sum_{n=0}^{+\infty }q^{2n+1}
f^{\sigma^{2n+1}}\left( \sigma \left( x\right) \right)  \\
& \qquad -\left[ \sigma^{-1}\left( \sigma ^{-1}\left( x\right) \right)
-\sigma \left( \sigma ^{-1}\left( x\right) \right) \right]
\sum_{n=0}^{+\infty} q^{2n+1}f^{\sigma ^{2n+1}}\left(
\sigma ^{-1}\left( x\right) \right) \bigg\}\\
& =\sum_{n=0}^{+\infty }q^{2n}f^{\sigma ^{2n}}\left( x\right)
-\sum_{n=0}^{+\infty }q^{2n+2}f^{\sigma ^{2n+2}}\left(x\right)\\
&= \sum_{n=0}^{+\infty }\bigg{(}q^{2n}f^{\sigma ^{2n}}\left( x\right)
-q^{2(n+1)}f^{\sigma ^{2(n+1)}}\left( x\right)\bigg{)}\\ & =f\left( x\right)\text{.}
\end{align*}
If $x=\omega _{0}$, then
\begin{align*}
\tilde{D}&_{q,\omega }\left[F\right]\left( \omega _{0}\right)\\
&= \lim_{h\rightarrow 0} \frac{F\left( \omega _{0}+h\right)
-F\left( \omega _{0}\right) }{h} \\
&=\lim_{h\rightarrow 0}\frac{1}{h}\left[ \sigma ^{-1}\left( \omega _{0}+h\right)
-\sigma \left( \omega _{0}+h\right) \right] \sum_{n=0}^{+ \infty}
q^{2n+1}f^{\sigma ^{2n+1}}\left( \omega _{0}+h\right)  \\
&= \lim_{h\rightarrow 0}\frac{1}{h}\left[q^{-1}\left(\omega_{0}
+h-\omega\right)-q\left(\omega_{0}+h\right)-\omega\right]\sum_{n=0}^{+ \infty }q^{2n+1}
f^{\sigma ^{2n+1}}\left( \omega _{0}+h\right)  \\
&= \lim_{h\rightarrow 0}\frac{1}{h}\left[\left(q^{-1}-q\right)\omega_{0}
+\left(-q^{-1}-1\right)\omega+\left(q^{-1}-q\right)h\right]
\sum_{n=0}^{+\infty } q^{2n+1}f^{\sigma^{2n+1}}\left( \omega _{0}+h\right)\\
&= \lim_{h\rightarrow 0}\frac{1}{h}\left[\frac{\left(q^{-1}-q\right)\omega}{1-q}
+\left(-q^{-1}-1\right)\omega+\left(q^{-1}-q\right)h\right]\sum_{n=0}^{+\infty}
q^{2n+1}f^{\sigma ^{2n+1}}\left( \omega _{0}+h\right)  \\
&= \lim_{h\rightarrow 0}\frac{1}{h}\left[\left(\frac{1+q}{q}
+\frac{-1-q}{q}\right)\omega+\left(q^{-1}-q\right)h\right]
\sum_{n=0}^{+\infty} q^{2n+1}f^{\sigma ^{2n+1}}\left( \omega_{0}+h\right)\\
&= \lim_{h\rightarrow 0}\frac{1-q^{2}}{q}\sum_{n=0}^{+\infty}
q^{2n+1}f^{\sigma ^{2n+1}}\left( \omega _{0}+h\right)  \\
&= \left( 1-q^{2}\right) \sum_{n=0}^{+\infty }q^{2n}
f\left( \omega_{0}\right)  \\
&= \left( 1-q^{2}\right) \frac{1}{1-q^{2}}f\left( \omega _{0}\right)\\
&= f\left(\omega_{0}\right) .
\end{align*}
Finally, since for $x\in I\backslash \{\omega_{0}\}$ we have
\begin{equation*}
\begin{split}
\int_{\omega _{0}}^{x}  \tilde{D}_{q,\omega }\left[ f\right] \left( t\right)
& \tilde{d}_{q,\omega }t
= \left[ \sigma ^{-1}\left( x\right)
-\sigma \left( x\right)\right] \sum_{n=0}^{+\infty }q^{2n+1}
\tilde{D}_{q,\omega }\left[ f \right]^{\sigma ^{2n+1}}\left( x\right)  \\
&= \left[ \sigma ^{-1}\left( x\right) -\sigma\left( x\right) \right]
\sum_{n=0}^{+\infty }q^{2n+1}\frac{f^{\sigma }\left( \sigma^{2n+1}\left( x\right) \right)
-f^{\sigma ^{-1}}\left( \sigma ^{2n+1}\left(x\right) \right) }{\sigma \left(
\sigma^{2n+1}\left( x\right) \right) -\sigma ^{-1}\left( \sigma ^{2n+1}\left( x\right) \right) } \\
&= \left[ \sigma ^{-1}\left(x\right) -\sigma \left( x\right) \right]
\sum_{n=0}^{+\infty }q^{2n+1}\frac{f^{\sigma }\left( \sigma ^{2n+1}\left( x\right) \right)
-f^{\sigma ^{-1}}\left( \sigma^{2n+1}\left( x\right) \right) }{q^{2n+1}\left(\sigma \left(x\right)
- \sigma ^{-1}\left( x\right)\right) }\\ & =\sum_{n=0}^{+\infty }\left[ f^{\sigma ^{2n}}\left(x\right)
-f^{\sigma^{2\left( n+1\right) }}\left( x\right) \right]  \\
&= f\left( x\right) -f\left( \omega _{0}\right),
\end{split}
\end{equation*}
where in the third equality we use Lemma~\ref{q^n}, then
\begin{align*}
\int_{a}^{b}\tilde{D}_{q,\omega }\left[ f\right] \left( t\right)
\tilde{d}_{q,\omega }t& =\int_{\omega_{0}}^{b}\tilde{D}_{q,\omega }\left[ f\right] \left(
t\right) \tilde{d}_{q,\omega }t-\int_{\omega_{0}}^{a}\tilde{D}_{q,\omega }\left[
f \right] \left( t\right) \tilde{d}_{q,\omega }t \\ & =f\left( b\right)
-f\left(a\right) \text{.}
\end{align*}
\end{proof}

The Hahn symmetric integral has the following properties.

\begin{theorem}
Let $f,g$ $:I\rightarrow \mathbb{R}$ be Hahn's symmetric integrable
on $I$, $a,b,c\in I$, and $\alpha ,\beta \in \mathbb{R}$. Then,
\begin{enumerate}
\item $\int_{a}^{a}f\left( t\right) \tilde{d}_{q,\omega}t=0$;

\item $\int_{a}^{b}f\left( t\right) \tilde{d}_{q,\omega}t
=-\int_{b}^{a}f \left( t\right) \tilde{d}_{q,\omega}t$;

\item $\int_{a}^{b}f\left( t\right) \tilde{d}_{q,\omega}t
=\int_{a}^{c} f\left(t\right) \tilde{d}_{q,\omega}t
+\int_{c}^{b}f\left( t\right)\tilde{d}_{q,\omega}t$;

\item $\int_{a}^{b}\left( \alpha f
+\beta g\right) \left( t\right) \tilde{d}_{q,\omega}t
=\alpha\int_{a}^{b}f\left( t\right) \tilde{d}_{q,\omega}t
+\beta \int_{a}^{b}g\left( t\right) \tilde{d}_{q,\omega}t$;

\item if $\tilde{D}_{q,\omega }\left[ f\right] $ and
$\tilde{D}_{q,\omega } \left[ g\right]$
are continuous at $\omega _{0}$, then
\begin{equation}
\label{eq:int:parts}
\int_{a}^{b}f^{\sigma^{-1}}\left( t\right) \tilde{D}_{q,\omega }\left[
g \right] \left( t\right) \tilde{d}_{q,\omega }t=f\left( t\right)
g\left( t\right) \bigg|_{a}^{b}-\int_{a}^{b}\tilde{D}_{q,\omega }\left[
f\right] \left( t\right) g^{\sigma }\left( t\right) \tilde{d}_{q,\omega }t.
\end{equation}
\end{enumerate}
\end{theorem}

\begin{proof}
Properties 1 to 4 are trivial. Property~5 follows from
Theorem~\ref{props derivada} and Theorem~\ref{Fundamental}: since
$$
\tilde{D}_{q,\omega }\left[ fg\right] \left( t\right)
=\tilde{D}_{q,\omega}\left[ f\right] \left( t\right)
g^{\sigma }\left( t\right) +f^{\sigma^{-1}}\left(
t\right) \tilde{D}_{q,\omega }\left[ g\right] \left( t\right),
$$
then
$$
f^{\sigma ^{-1}}\left( t\right) \tilde{D}_{q,\omega }\left[ g \right] \left( t\right)
=\tilde{D}_{q,\omega }\left[ fg\right] \left( t\right)
-\tilde{D}_{q,\omega }\left[ f\right] \left( t\right)
g^{\sigma}\left( t\right)
$$
and hence,
$$
\int_{a}^{b}f^{\sigma ^{-1}}\left( t\right) \tilde{D}_{q,\omega }\left[
g\right] \left( t\right) \tilde{d}_{q,\omega }t=f\left( t\right) g\left( t\right)
\bigg|_{a}^{b}-\int_{a}^{b}\tilde{D}_{q,\omega } \left[ f\right] \left( t\right)
g^{\sigma }\left( t\right) \tilde{d}_{q,\omega }t.
$$
\end{proof}

\begin{remark}
Relation \eqref{eq:int:parts} gives a Hahn's
symmetric integration by parts formula.
\end{remark}

\begin{remark}
Using Lemma~\ref{composta} and the Hahn symmetric integration
by parts formula \eqref{eq:int:parts}, we conclude that
\begin{equation}
\label{partes}
\int_{a}^{b} f\left( t\right) \tilde{D}_{q,\omega }\left[ g\right] \left( t\right)
\tilde{d}_{q,\omega }t=f^{\sigma }\left( t\right) g\left(t\right)
\bigg|_{a}^{b}-q\int_{a}^{b}\left( \tilde{D}_{q,\omega }\left[
f\right] \right)^{\sigma }\left( t\right) g^{\sigma }\left( t\right)
\tilde{d}_{q,\omega}t.
\end{equation}
\end{remark}

\begin{proposition}
\label{desigualdade}
Let $c\in I$, $f$ and $g$ be Hahn's symmetric integrable on $I$.
Suppose that $\left\vert f\left( t\right) \right\vert \leq g\left( t\right)$
for all $t\in \left\{ \sigma ^{2n+1}\left( c\right) :
n\in \mathbb{N}_{0}\right\} \cup \left\{ \omega _{0}\right\}$.
\begin{enumerate}
\item If $c\geq \omega _{0}$, then
\begin{equation*}
\left\vert \int_{\omega _{0}}^{c}f\left( t\right)
\tilde{d}_{q,\omega }t\right\vert
\leq \int_{\omega _{0}}^{c}g\left( t\right) \tilde{d}_{q,\omega}t.
\end{equation*}

\item If $c<\omega _{0}$, then
\begin{equation*}
\left\vert \int_{c}^{\omega _{0}}f\left( t\right) \tilde{d}_{q,\omega }t\right\vert
\leq \int_{c}^{\omega _{0}}g\left( t\right) \tilde{d}_{q,\omega }t.
\end{equation*}
\end{enumerate}
\end{proposition}

\begin{proof}
If $c\geq \omega_{0}$, then
\begin{equation*}
\begin{split}
\left\vert \int_{\omega _{0}}^{c}f\left( t\right) \tilde{d}_{q,\omega} t\right\vert
&=\left\vert \left[ \sigma ^{-1}\left( c\right) -\sigma \left( c\right)\right]
\sum_{n=0}^{+\infty }q^{2n+1} f^{\sigma^{2n+1}}\left( c\right) \right\vert\\
&\leq \left[ \sigma ^{-1}\left( c\right) -\sigma \left( c\right) \right]
\sum_{n=0}^{+\infty }q^{2n+1}\left\vert f^{\sigma ^{2n+1}}\left( c\right) \right\vert \\
&\leq \left[ \sigma ^{-1}\left( c\right) -\sigma \left( c\right) \right]
\sum_{n=0}^{+\infty }q^{2n+1}g^{\sigma ^{2n+1}}\left( c\right)\\
&=\int_{\omega_{0}}^{c}g\left( t\right) \tilde{d}_{q,\omega }t.
\end{split}
\end{equation*}
If $c<\omega_{0}$, then
\begin{equation*}
\begin{split}
\left\vert \int_{c}^{\omega _{0}} f\left( t\right) \tilde{d}_{q,\omega}t\right\vert
&= \left\vert -\left[ \sigma ^{-1}\left( c\right) -\sigma \left( c\right)\right]
\sum_{n=0}^{+\infty }q^{2n+1} f^{\sigma^{2n+1}}\left( c\right) \right\vert\\
&\leq \left\vert \sigma ^{-1}\left( c\right) -\sigma \left( c\right)
\right\vert \sum_{n=0}^{+\infty }q^{2n+1}\left\vert f^{\sigma ^{2n+1}}\left(
c\right) \right\vert\\
&=-\left[ \sigma ^{-1}\left( c\right) -\sigma \left( c\right)\right]
\sum_{n=0}^{+\infty }q^{2n+1}\left\vert f^{\sigma ^{2n+1}}\left( c\right) \right\vert\\
&\leq -\left[ \sigma ^{-1}\left( c\right) -\sigma \left( c\right)\right]
\sum_{n=0}^{+\infty }q^{2n+1}g^{\sigma ^{2n+1}}\left( c\right)\\
&=-\int_{\omega _{0}}^{c}g\left( t\right) \tilde{d}_{q,\omega }t,\\
&=\int_{c}^{\omega_{0}}g\left( t\right) \tilde{d}_{q,\omega }t.
\end{split}
\end{equation*}
providing the desired equality.
\end{proof}

As an immediate consequence, we have the following result.

\begin{corollary}
Let $c\in I$ and $f$ be Hahn's symmetric integrable on $I$.
Suppose that $f\left( t\right) \geq 0$ for all
$t\in \left\{ \sigma ^{2n+1}\left( c\right) :n\in \mathbb{N}_{0}\right\}
\cup \left\{ \omega _{0}\right\}$.
\begin{enumerate}
\item If $c\geq \omega_{0}$, then
\begin{equation*}
\int_{\omega _{0}}^{c}f\left( t\right) \tilde{d}_{q,\omega }t\geq 0\text{.}
\end{equation*}
\item If $c<\omega_{0}$, then
\begin{equation*}
\int_{c}^{\omega _{0}}f\left( t\right) \tilde{d}_{q,\omega }t\geq 0\text{.}
\end{equation*}
\end{enumerate}
\end{corollary}

\begin{remark}
In general it is not true that if $f$ is a nonnegative function
on $\left[ a,b\right]$, then
\begin{equation*}
\int_{a}^{b}f\left( t\right) \tilde{d}_{q,\omega }t\geq 0\text{.}
\end{equation*}
As an example, consider the function $f$ defined in $\left[ -5,5\right]$ by
\begin{equation*}
f\left( x\right) =\left\{
\begin{array}{ccc}
6 & \text{\ \ if \ } & t=3 \\ 1 & \text{\ \ if \ }
& t=4 \\ 0 & \text{\ \ if \ } & t\in \left[ -5,5\right]
\backslash \left\{ 3, 4\right\}.
\end{array}
\right.
\end{equation*}
For $q=\frac{1}{2}$ and $\omega =1$, this function is Hahn's symmetric
integrable because is continuous at $\omega_{0}=2$. However,
\begin{align*}
\int_{4}^{6} & f\left( t\right) \tilde{d}_{q,\omega }t
=\int_{2}^{6}f\left( t\right) \tilde{d}_{q,\omega }t
-\int_{2}^{4} f\left( t\right) \tilde{d}_{q,\omega }t\\
&= \left(10-4\right)\sum_{n=0}^{+\infty }\left( \frac{1}{2}\right)^{2n+1}
f^{\sigma^{2n+1}}\left( 6\right) -\left(6-3\right)
\sum_{n=0}^{+\infty }\left(\frac{1}{2}\right)^{2n+1}
f^{\sigma^{2n+1}}\left( 4\right)\\
&= 6\left( \frac{1}{2}\right) \times 1-3\left( \frac{1}{2}\right)
\times 6\\
&= -6.
\end{align*}
This example also proves that, in general, it is not true that
\begin{equation*}
\left\vert \int_{a}^{b}f\left( t\right) \tilde{d}_{q,\omega }t\right\vert
\leq \int_{a}^{b}\left\vert f\left( t\right) \right\vert \tilde{d}_{q,\omega}t
\end{equation*}
for any $a,b\in I$.
\end{remark}


\section{Hahn's symmetric variational calculus}
\label{M}

We begin this section with some useful definitions and notations.
For $s\in I$ we set
\begin{equation*}
\left[ s\right] _{q,\omega }:=\left\{ \sigma ^{2n+1}\left( s\right)
: n\in \mathbb{N}_{0}\right\} \cup \left\{ \omega _{0}\right\} .
\end{equation*}
Let $a,b\in I$ with $a<b$.
We define the Hahn symmetric interval from $a$ to $b$ by
\begin{equation*}
\left[ a,b\right] _{q,\omega }:=\left\{ \sigma ^{2n+1}\left( a\right)
: n\in \mathbb{N}_{0}\right\} \cup \left\{ \sigma ^{2n+1}\left( b\right)
: n\in \mathbb{N}_{0}\right\} \cup \left\{ \omega _{0}\right\} ,
\end{equation*}
that is,
\begin{equation*}
\left[ a,b\right] _{q,\omega }=\left[ a\right]_{q,\omega }
\cup \left[ b \right]_{q,\omega }.
\end{equation*}
Let $r \in \{0,1\}$. We denote the linear space
\begin{equation*}
\left\{ y:I\rightarrow \mathbb{R}\  |\  \tilde{D}_{q,\omega }^{i}\left[ y\right] ,
i = 0,r, \text{ are bounded on }\left[a,b\right]_{q,\omega } \text{ and continuous at }
\omega _{0}\right\}
\end{equation*}
endowed with the norm
\begin{equation*}
\left\Vert y\right\Vert _{r}=\sum_{i=0}^{r}\sup_{t\in \left[ a,b\right]_{q,\omega }}\left\vert
\tilde{D}_{q,\omega }^{i}\left[ y\right] \left( t\right)\right\vert,
\end{equation*}
where $\tilde{D}_{q,\omega }^{0}\left[ y\right]=y$,
by $\mathcal{Y}^{r}\left( \left[ a,b\right]_{q,\omega }, \mathbb{R}\right)$.

\begin{definition}
We say that $y$ is an admissible function to problem \eqref{P} if
$y\in\mathcal{Y}^{1}\left( \left[ a,b\right] _{q,\omega }, \mathbb{R}\right)$
and $y$ satisfies the boundary conditions $y\left( a\right)=\alpha$
and $y\left( b\right) =\beta$.
\end{definition}

\begin{definition}
We say that $y_{\ast }$ is a local minimizer (resp. local maximizer)
to problem \eqref{P} if $y_{\ast }$ is an admissible function
and there exists $\delta >0$ such that
\begin{equation*}
\mathcal{L}\left( y_{\ast }\right) \leq \mathcal{L}\left( y\right)
\quad \text{ (resp. }\mathcal{L}\left( y_{\ast }\right)
\geq \mathcal{L}\left( y\right)\text{)}
\end{equation*}
for all admissible $y$ with $\left\Vert y_{\ast }
-y\right\Vert _{1}<\delta$.
\end{definition}

\begin{definition}
We say that $\eta \in \mathcal{Y}^{1}\left(\left[ a,b\right]_{q,\omega},
\mathbb{R}\right)$ is an admissible variation to problem \eqref{P}
if $\eta \left(a\right) =0=\eta \left( b\right)$.
\end{definition}

Before proving our main results, we begin with three basic lemmas.


\subsection{Basic Lemmas}

The following results are useful to prove Theorem~\ref{Euler}.

\begin{lemma}[Fundamental lemma of the Hahn symmetric variational calculus]
\label{fundamental}
Let $f\in \mathcal{Y}^{0}\left(\left[ a,b\right]_{q,\omega}, \mathbb{R}\right)$.
One has
\begin{equation*}
\int_{a}^{b}f\left( t\right) h^{\sigma }\left( t\right)
\tilde{d}_{q,\omega}t=0
\end{equation*}
for all $h\in \mathcal{Y}^{0}\left(\left[ a,b\right]_{q,\omega}, \mathbb{R}\right)$
with $h\left( a\right) =h\left( b\right) =0$ if, and only if,
$f\left(t\right) =0$ for all $t\in \left[ a,b\right]_{q,\omega }$.
\end{lemma}

\begin{proof}
The implication \textquotedblleft$\Leftarrow$\textquotedblright\ is obvious.
Let us prove the implication \textquotedblleft$\Rightarrow$\textquotedblright.
Suppose, by contradiction, that exists $p\in\left[ a,b\right]_{q,\omega}$
such that $f\left( p\right) \neq 0$.
\begin{enumerate}
\item If $p\neq\omega_{0}$, then $p=\sigma^{2k+1}\left( a\right)$
or $p=\sigma^{2k+1}\left( b\right)$ for some $k \in \mathbb{N}_{0}$.
\begin{enumerate}
\item Suppose that $a\neq \omega_{0}$ and $b\neq \omega_{0}$.
In this case we can assume, without loss of generality, that
$p=\sigma^{2k+1}\left(a\right)$. Define
\begin{equation*}
h\left( t\right) =
\begin{cases}
f^{\sigma^{2k+1}}\left( a\right)  & \text{if } t=\sigma^{2k+2}\left(a\right)\\
0 & \text{otherwise.}
\end{cases}
\end{equation*}
Then,
\begin{align*}
\int_{a}^{b} & f\left( t\right) h^{\sigma }\left( t\right) \tilde{d}_{q,\omega}t \\
&= \left[ \sigma ^{-1}\left( b\right) -\sigma \left( b\right) \right]
\sum_{n=0}^{+\infty }q^{2n+1}f^{\sigma ^{2n+1}}\left( b\right)
h^{\sigma^{2n+2}}\left( b\right)\\ &\qquad -\left[ \sigma ^{-1}\left( a\right)
-\sigma\left( a\right) \right] \sum_{n=0}^{+\infty }q^{2n+1}
f^{\sigma^{2n+1}}\left( a\right) h^{\sigma^{2n+2}}\left( a\right)\\
&= -\left[ \sigma ^{-1}\left(a\right) -\sigma \left( a\right) \right]
q^{2k+1}\left[ f^{\sigma ^{2k+1}}\left( a\right) \right] ^{2}\neq 0,
\end{align*}
which is a contradiction.

\item Suppose that $a\neq\omega_{0}$ and $b=\omega_{0}$. Therefore,
$p=\sigma^{2k+1}\left( a\right) $ for some $k\in \mathbb{N}_{0}$. Define
\begin{equation*}
h\left( t\right) =
\begin{cases}
f^{\sigma^{2k+1}}\left( a\right) & \text{if } t
=\sigma^{2k+2}\left(a\right)\\ 0 & \text{otherwise.}
\end{cases}
\end{equation*}
We obtain a contradiction with a similar proof as in case (a).

\item The case $a=\omega_{0}$ and $b\neq\omega_{0}$ is similar to (b).
\end{enumerate}

\item If $p=\omega_{0}$, we assume, without loss of generality,
that $f\left(p\right) >0$. Since
\begin{equation*}
\lim_{n\rightarrow +\infty }\sigma^{2k+2}\left( a\right)
=\lim_{n\rightarrow +\infty }\sigma ^{2k+2}\left( b\right)
=\omega_{0}
\end{equation*}
and $f$ is continuous at $\omega_{0}$,
\begin{equation*}
\lim_{n\rightarrow +\infty }f^{\sigma ^{2k+1}}\left( a\right)
=\lim_{n\rightarrow +\infty }f^{\sigma ^{2k+1}}\left( b\right)
=f\left( \omega _{0}\right).
\end{equation*}
Therefore, there exists an order $n_{0}\in \mathbb{N}$
such for all $n>n_{0}$ the inequalities
\begin{equation*}
f^{\sigma ^{2k+1}}\left( a\right) >0 \ \text{ and }\
f^{\sigma^{2k+1}}\left(b\right) >0
\end{equation*}
hold.

\begin{enumerate}
\item If $a,b\neq \omega_{0}$, then for some $k>n_{0}$ we define
\begin{equation*}
h\left( t\right) =
\begin{cases}
-\frac{f^{\sigma ^{2k+1}}\left( b\right)}{\sigma ^{-1}\left( a\right)
-\sigma \left( a\right) } & \text{if } t=\sigma ^{2k+2}\left( a\right)\\
\frac{f^{\sigma ^{2k+1}}\left( a\right) }{\sigma ^{-1}\left( b\right)
-\sigma \left( b\right) } & \text{if } t=\sigma ^{2k+2}\left( b\right)\\ 0
& \text{otherwise.}
\end{cases}
\end{equation*}
Hence,
\begin{equation*}
\int_{a}^{b}f\left( t\right) h^{\sigma }\left( t\right) \tilde{d}_{q,\omega}t
=2q^{2k+1}f^{\sigma ^{2k+1}}\left( a\right) f^{\sigma^{2k+1}}\left(b\right) > 0.
\end{equation*}

\item If $a=\omega_{0}$, then we define
\begin{equation*}
h\left( t\right) =
\begin{cases}
f^{\sigma^{2k+1}}\left( b\right)  & \text{if } t
=\sigma ^{2k+2}\left(b\right)\\ 0 & \text{otherwise.}
\end{cases}
\end{equation*}
Therefore,
\begin{equation*}
\int_{\omega _{0}}^{b}f\left( t\right) h^{\sigma }\left( t\right) \tilde{d}_{q,\omega}t
=\left[ \sigma ^{-1}\left( b\right) -\sigma \left( b\right)\right] q^{2k+1}\left[
f^{\sigma ^{2k+1}}\left( b\right) \right] ^{2}\neq 0.
\end{equation*}
\item If $b=\omega_{0}$, the proof is similar to the previous case.
\end{enumerate}
\end{enumerate}
\end{proof}

\begin{definition}
Let $s\in I$ and $g:I\times \left] -\bar{\theta},\bar{\theta}\right[
\rightarrow \mathbb{R}$. We say that $g\left( t,\cdot \right)$
is differentiable at $\theta_{0}$ uniformly in $\left[ s\right]_{q,\omega }$
if, for every $\varepsilon >0$, there exists $\delta >0$ such that
\begin{equation*}
0<\left\vert \theta -\theta _{0}\right\vert <\delta \Rightarrow \left\vert
\frac{g\left( t,\theta \right) -g\left( t,\theta _{0}\right) }{ \theta
-\theta_{0}}-\partial _{2}g\left( t,\theta _{0}\right) \right\vert < \varepsilon
\end{equation*}
for all $t\in \left[ s\right] _{q,\omega }$,
where $\partial_{2}g =\frac{\partial g}{\partial \theta }$.
\end{definition}

\begin{lemma}[\textrm{cf.} \cite{Malinowska}]
\label{tecnico}
Let $s\in I$ and assume that $g:I\times \left] -\bar{\theta},\bar{\theta}\right[
\rightarrow \mathbb{R}$ is differentiable at $\theta _{0}$ uniformly in
$\left[ s\right]_{q,\omega }$. If $\int_{\omega _{0}}^{s}g\left( t,\theta _{0}\right)
\tilde{d}_{q,\omega }t$ exists, then $G\left( \theta \right) :=\int_{\omega_{0}}^{s}
g\left( t,\theta \right) \tilde{d}_{q,\omega }t$ for $\theta$ near $\theta _{0}$,
is differentiable at $\theta _{0}$ with
\begin{equation*}
G^{\prime }\left( \theta _{0}\right)
=\int_{\omega _{0}}^{s}\partial_{2} g\left( t,\theta _{0}\right) \tilde{d}_{q,\omega }t.
\end{equation*}
\end{lemma}

\begin{proof}
For $s>\omega_{0}$ the proof is similar to the proof given in Lemma~3.2 of \cite{Malinowska}.
The result is trivial for $s = \omega_{0}$. Suppose that  $s<\omega_{0}$ and let
$\varepsilon >0$ be arbitrary. Since $g\left( t,\cdot \right)$ is differentiable
at $\theta_{0}$ uniformly in $\left[ s\right] _{q,\omega }$, then there exists $\delta >0$
such that for all $t\in \left[ s\right]_{q,\omega }$ and for $0<\left\vert \theta
-\theta _{0}\right\vert <\delta$ the following inequality holds:
\begin{equation}
\label{treta}
\left\vert \frac{g\left( t,\theta \right) -g\left( t,\theta _{0}\right) }{
\theta -\theta _{0}}-\partial _{2}g\left( t,\theta _{0}\right)
\right\vert < \frac{\varepsilon }{2\left(\omega _{0}- s\right)}.
\end{equation}
Since, for $0<\left\vert \theta -\theta_{0}\right\vert <\delta $, we have
\begin{align*}
\Biggl\vert &\frac{G\left( \theta \right)-G\left( \theta _{0}\right) }{\theta
-\theta _{0}}-\int_{\omega _{0}}^{s}\partial _{2}g\left( t,\theta_{0}\right)
\tilde{d}_{q,\omega }t \Biggr\vert \\ &=\left\vert \frac{\int_{\omega_{0}}^{s}
g\left( t,\theta \right) \tilde{d}_{q,\omega}t -\int_{\omega _{0}}^{s}
g\left(t,\theta _{0}\right) \tilde{d}_{q,\omega }t}{\theta -\theta_{0}}
-\int_{\omega _{0}}^{s}\partial_{2} g\left( t,\theta_{0}\right)
\tilde{d}_{q,\omega }t\right\vert \\
&= \left\vert \int_{\omega _{0}}^{s}\left[ \frac{g\left( t,\theta \right)
- g\left( t,\theta _{0}\right) }{\theta -\theta _{0}}
-\partial_{2}g\left(t,\theta_{0}\right) \right] \tilde{d}_{q,\omega }t\right\vert\\
& < \int_{s}^{\omega _{0}}\frac{\varepsilon }{2\left(\omega_{0}- s\right)} \tilde{d}_{q,\omega }t
\text{\ \ (using Proposition~\ref{desigualdade} and inequality \eqref{treta})}\\
&= \frac{\varepsilon }{2\left(\omega _{0}- s\right) }\int_{s}^{\omega_{0}}1
\tilde{d}_{q,\omega }t =\frac{\varepsilon }{2}<\varepsilon
\end{align*}
then we can conclude that $\displaystyle G^{\prime }\left( \theta \right)
=\int_{\omega _{0}}^{s}\partial _{2}g\left( t,\theta_{0}\right)
\tilde{d}_{q,\omega}t$.
\end{proof}

For an admissible variation $\eta$ and an admissible function $y$,
we define $\phi :\left] -\bar{\epsilon},\bar{\epsilon}\right[ \rightarrow \mathbb{R}$
by $\phi\left( \epsilon \right) := \mathcal{L}\left( y+\epsilon \eta \right)$.
The first variation of functional $\mathcal{L}$ of problem \eqref{P} is defined by
$\delta \mathcal{L}\left( y,\eta \right) :=\phi ^{\prime }\left( 0\right)$. Note that
\begin{align*}
\mathcal{L}\left( y+\epsilon \eta \right) &= \int_{a}^{b}
L\left( t,y^{\sigma}\left( t\right)
+ \epsilon \eta^{\sigma }\left( t\right), \tilde{D}_{q,\omega }\left[
y\right] \left( t\right) +\epsilon \tilde{D}_{q,\omega } \left[ \eta \right]
\left( t\right) \right) \tilde{d}_{q,\omega }t \\
&= \mathcal{L}_{b}\left( y+\epsilon \eta \right)
-\mathcal{L}_{a}\left( y+\epsilon \eta \right) ,
\end{align*}
where
\begin{equation*}
\mathcal{L}_{\xi }\left( y+\epsilon \eta \right)
= \int_{\omega _{0}}^{\xi}L\left( t,y^{\sigma }\left( t\right)
+\epsilon \eta ^{\sigma }\left(t\right), \tilde{D}_{q,\omega }\left[ y\right]\left( t\right)
+\epsilon \tilde{D}_{q,\omega }\left[ \eta \right]\left(t\right)\right)\tilde{d}_{q,\omega } t
\end{equation*}
with $\xi \in \left\{ a,b\right\}$. Therefore, $\delta \mathcal{L}\left( y,\eta \right)
= \delta \mathcal{L}_{b}\left( y,\eta\right) -\delta \mathcal{L}_{a}\left(y,\eta \right)$.

The following lemma is a direct consequence of Lemma~\ref{tecnico}.

\begin{lemma}
\label{tecnico 2}
For an admissible variation $\eta $ and an admissible function $y$, let
\begin{equation*}
g\left( t,\epsilon \right) :=L\left( t,y^{\sigma }\left( t\right)
+\epsilon \eta ^{\sigma }\left( t\right),
\tilde{D}_{q,\omega }\left[ y\right] \left( t\right)
+\epsilon \tilde{D}_{q,\omega }\left[\eta\right]\left( t\right) \right) .
\end{equation*}
Assume that
\begin{enumerate}
\item $g\left( t,\cdot \right) $ is differentiable at $\omega_{0}$
uniformly in $\left[ a,b\right] _{q,\omega }$;

\item $\mathcal{L}_{\xi}\left( y+\epsilon \eta \right)
=\int_{\omega_{0}}^{\xi}g\left( t,\epsilon \right) \tilde{d}_{q,\omega }t$,
$\xi \in \left\{a,b\right\}$, exist for $\epsilon \approx 0$;

\item $\int_{\omega _{0}}^{a}\partial_{2}g\left( t,0\right) \tilde{d}_{q,\omega }t$
and $\int_{\omega_{0}}^{b}\partial_{2}g\left( t,0\right) \tilde{d}_{q,\omega }t$ exist.
\end{enumerate}
Then,
\begin{multline*}
\phi^{\prime }\left( 0\right) :=\delta \mathcal{L}\left( y,\eta \right)
= \int_{a}^{b}\Biggl[\partial _{2}L\left( t,y^{\sigma }\left( t\right),
\tilde{D}_{q,\omega }\left[ y\right] \left( t\right) \right)
\eta^{\sigma }\left(t\right)\\ +\partial _{3} L\left( t,y^{\sigma }\left( t\right),
\tilde{D}_{q,\omega} \left[ y\right] \left( t\right) \right)
\tilde{D}_{q,\omega }\left[\eta\right]\left(t\right) \Biggr]\tilde{d}_{q,\omega }t.
\end{multline*}
\end{lemma}


\subsection{Optimality Conditions}
\label{sec:o:c}

In this section we present a necessary optimality condition
(the Hanh symmetric Euler--Lagrange equation)
and a sufficient optimality condition to problem \eqref{P}.

\begin{theorem}[The Hahn symmetric Euler--Lagrange equation]
\label{Euler}
Under hypotheses (H1)--(H3) and conditions 1 to 3 of Lemma~\ref{tecnico 2}
on the Lagrangian $L$, if $y_{\ast}\in \mathcal{Y}^{1}\left(\left[a,b\right]_{q,\omega},
\mathbb{R}\right)$ is a local extremizer to problem \eqref{P}, then $y_{\ast }$
satisfies the Hahn symmetric Euler--Lagrange equation
\begin{equation}
\label{EqEuler}
\partial_{2}L\left( t,y^{\sigma }\left( t\right), \tilde{D}_{q,\omega}
\left[ y\right] \left( t\right) \right) =\tilde{D}_{q,\omega }\left[\tau
\mapsto \partial_{3}L\left( \sigma\left(\tau \right),
y^{\sigma ^{2}}\left( \tau \right), \left(\tilde{D}_{q,\omega}\left[
y\right] \right)^{\sigma }\left( \tau\right) \right) \right] \left( t\right)
\end{equation}
for all $t\in \left[ a,b\right]_{q,\omega}$.
\end{theorem}

\begin{proof}
Let $y_{\ast}$ be a local minimizer (resp. maximizer) to problem \eqref{P}
and $\eta $ an admissible variation. Define $\phi : \mathbb{R} \rightarrow \mathbb{R}$
by $\phi \left( \epsilon \right) :=\mathcal{L}\left( y_{\ast} +\epsilon \eta \right)$.
A necessary condition for $y_{\ast}$ to be an extremizer is given by
$\phi^{\prime }\left( 0\right) =0$. By Lemma~\ref{tecnico 2},
\begin{multline*}
\int_{a}^{b}\Biggl[ \partial_{2}L\left( t,y_{\ast}^{\sigma }\left( t\right),
\tilde{D}_{q,\omega }\left[ y_{\ast}\right] \left( t\right) \right) \eta^{\sigma}\left(t\right) \\
+\partial_{3}L\left( t,y_{\ast}^{\sigma }\left( t\right), \tilde{D}_{q,\omega }\left[
y_{\ast}\right] \left( t\right) \right) \tilde{D}_{q,\omega}\left[\eta\right] \left( t\right)
\Biggr] \tilde{d}_{q,\omega }t =0.
\end{multline*}
Using the integration by parts formula \eqref{partes}, we get
\begin{align*}
\int_{a}^{b} & \partial _{3}L\left( t,y_{\ast}^{\sigma }\left( t\right),
\tilde{D}_{q,\omega }\left[ y_{\ast}\right] \left( t\right) \right)
\tilde{D}_{q,\omega}\left[\eta\right] \left( t\right) \tilde{d}_{q,\omega }t \\
& =\partial _{3}L\left( \sigma \left( t\right), y_{\ast}^{\sigma ^{2}}\left(t\right),
\left( \tilde{D}_{q,\omega }\left[ y_{\ast}\right] \right)^{\sigma}\left( t\right)
\right) \eta \left( t\right) \bigg|_{a}^{b}\\
&\qquad - q\int_{a}^{b}\left( \tilde{D}_{q,\omega }\left[\tau \mapsto \partial_{3}
L\left(\tau ,y_{\ast}^{\sigma }\left( \tau \right),
\left( \tilde{D}_{q,\omega }\left[ y_{\ast}\right] \right) \left( \tau \right)
\right) \right] \right)^{\sigma}\left( t\right) \eta^{\sigma }\left( t\right)
\tilde{d}_{q,\omega }t.
\end{align*}
Since $\eta \left( a\right) =\eta \left( b\right) =0$, then
\begin{multline*}
\int_{a}^{b}\bigg[\partial _{2} L\left( t,y_{\ast}^{\sigma }\left( t\right),
\tilde{D}_{q,\omega }\left[ y_{\ast}\right] \left( t\right) \right)\\
-q\left(\tilde{D}_{q,\omega }\left[ \tau \mapsto \partial_{3}
L\left(\tau, y_{\ast}^{\sigma}\left(\tau\right), \left( \tilde{D}_{q,\omega }\left[
y_{\ast}\right] \right) \left( \tau \right) \right) \right] \right)^{\sigma }\left( t\right)
\bigg] \eta^{\sigma }\left( t\right) \tilde{d}_{q,\omega }t = 0
\end{multline*}
and by Lemma~\ref{fundamental} we get
\begin{eqnarray*}
&\partial _{2}L\left( t,y_{\ast}^{\sigma }\left( t\right),
\tilde{D}_{q,\omega } \left[ y_{\ast}\right] \left( t\right) \right)
=& q\left( \tilde{D}_{q,\omega}\left[ \tau \mapsto \partial _{3}
L\left(\tau, y_{\ast}^{\sigma }\left(\tau\right), \tilde{D}_{q,\omega}\left[
y_{\ast}\right]\left(\tau\right)\right)\right]\right)^{\sigma }\left( t\right)
\end{eqnarray*}
for all $t\in \left[ a,b\right]_{q,\omega}$.
Finally, using Lemma~\ref{composta}, we conclude that
\begin{multline*}
\partial _{2}L\left( t,y_{\ast}^{\sigma }\left( t\right) ,
\tilde{D}_{q,\omega }\left[ y_{\ast}\right] \left( t\right) \right)\\
=\tilde{D}_{q,\omega}\left[ \tau \mapsto \partial_{3}
L\left( \sigma\left(\tau \right), y_{\ast}^{\sigma ^{2}}\left(\tau\right),
\left(\tilde{D}_{q, \omega}\left[y_{\ast}\right]\right)^{\sigma}\left(
\tau\right)\right)\right] \left( t\right).
\end{multline*}
\end{proof}

The particular case $\omega =0$ gives the $q$-symmetric Euler--Lagrange equation.

\begin{corollary}[The $q$-symmetric Euler--Lagrange equation \cite{Cruz2}]
\label{q-Euler}
Let $\omega =0$. Under hypotheses (H1)--(H3) and conditions 1 to 3
of Lemma~\ref{tecnico 2} on the Lagrangian $L$,
if $y_{\ast } \in \mathcal{Y}^{1}\left(\left[ a,b\right]_{q,0}, \mathbb{R}\right)$
is a local extremizer to problem \eqref{P} (with $\omega =0$),
then $y_{\ast }$ satisfies the $q$-symmetric Euler--Lagrange equation
\begin{equation*}
\partial_{2}L\left( t,y\left( qt\right) ,\tilde{D}_{q}\left[ y\right]
\left( t\right) \right) =\tilde{D}_{q}\left[ \tau \mapsto
\partial_{3}L\left( q\tau, y\left( q^{2}\tau \right),
\tilde{D}_{q}\left[ y\right] \left( q\tau\right) \right) \right] \left( t\right)
\end{equation*}
for all $t\in \left[ a,b\right]_{q}$.
\end{corollary}

To conclude this section, we prove a sufficient optimality condition to \eqref{P}.

\begin{definition}
Given a Lagrangian $L$, we say that $L\left( t,u,v\right)$ is jointly convex
(resp. concave) in $\left( u,v\right)$ if, and only if, $\partial_{i}L$, $i=2,3$,
exist and are continuous and verify the following condition:
\begin{equation*}
L\left( t,u+u_{1},v+v_{1}\right) -L\left( t,u,v\right)
\geq \text{ (resp. $\leq$) } \partial_{2}L\left( t,u,v\right) u_{1}
+\partial _{3}L\left(t,u,v\right) v_{1}
\end{equation*}
for all $\left( t,u,v\right) ,\left( t,u+u_{1},v+v_{1}\right)
\in I\times \mathbb{R}^{2}$.
\end{definition}

\begin{theorem}
\label{Suficiente}
Suppose that $a<b$ and $a,b\in \lbrack c]_{q,\omega}$ for some $c\in I$.
Also, assume that $L$ is a jointly convex (resp. concave) function in
$\left(u,v\right) $. If $y_{\ast }$ satisfies the Hahn symmetric
Euler--Lagrange equation \eqref{EqEuler}, then $y_{\ast }$
is a global minimizer (resp. maximizer) to problem \eqref{P}.
\end{theorem}

\begin{proof}
Let $L$ be a jointly convex function in $\left( u,v\right)$
(the concave case is similar). Then,
for any admissible variation $\eta$, we have
\begin{align*}
\mathcal{L}( y_{\ast} &+ \eta ) -\mathcal{L}\left( y_{\ast}\right) \\
&= \int_{a}^{b}\Biggl(L\left( t,y^{\sigma }\left( t\right)
+\eta ^{\sigma }\left( t\right),\tilde{D}_{q,\omega }\left[ y\right]
\left( t\right) +\tilde{D}_{q,\omega }\left[ \eta \right] \left( t\right) \right)\\
&\qquad\qquad -L\left(t,y^{\sigma}\left( t\right), \tilde{D}_{q,\omega }\left[
y\right] \left( t\right) \right) \Biggr)\tilde{d}_{q,\omega }t \\
&\geq \int_{a}^{b}\Biggl( \partial_{2} L\left( t,y^{\sigma }\left( t\right),
\tilde{D}_{q,\omega }\left[ y\right] \left( t\right) \right) \eta^{\sigma}\left( t\right)\\
&\qquad\qquad +\partial_{3} L\left( t,y^{\sigma }\left( t\right), \tilde{D}_{q,\omega }\left[
y\right] \left( t\right) \right) \tilde{D}_{q,\omega } \left[ \eta \right] \left(t\right)
\Biggr) \tilde{d}_{q,\omega }t.
\end{align*}
Using the integration by parts formula \eqref{partes} and Lemma~\ref{composta}, we get
\begin{align*}
\mathcal{L}\left( y_{\ast }+\eta \right) &-\mathcal{L}\left( y_{\ast }\right)
\geq \partial_{3}L\left( \sigma \left( t\right) ,y_{\ast }^{\sigma^{2}}\left(t\right),
\left( \tilde{D}_{q,\omega }\left[ y\right] \right)^{\sigma }\left( t\right) \right)
\eta \left( t\right) \bigg|_{a}^{b} \\
&+ \int_{a}^{b}\bigg[\partial _{2}L\left( t,y_{\ast }^{\sigma }\left(t\right),
\tilde{D}_{q,\omega }\left[ y_{\ast }\right] \left( t\right)\right)\\
&- \tilde{D}_{q,\omega }\left[\tau \mapsto \partial_{3}L\left( \sigma\left(\tau \right),
y^{\sigma ^{2}}\left(\tau\right), \left( \tilde{D}_{q,\omega }\left[ y\right]
\right)^{\sigma }\left(\tau\right) \right) \left( t\right) \right] \bigg]
\eta^{\sigma }\left( t\right) \tilde{d}_{q,\omega }t.
\end{align*}
Since $y_{\ast }$ satisfies \eqref{EqEuler} and $\eta $
is an admissible variation, we obtain
\begin{equation*}
\mathcal{L}\left( y_{\ast }+\eta \right) -\mathcal{L}\left( y_{\ast }\right) \geq 0,
\end{equation*}
proving that $y_{*}$ is a minimizer to problem \eqref{P}.
\end{proof}

\begin{ex}
Let $q\in \left] 0,1\right[ $ and $\omega \geq 0$ be  fixed real numbers. Also,
let $I \subseteq \mathbb{R}$ be an interval such that
$a:=\omega_{0},b\in I$ and $a<b$. Consider the problem
\begin{equation}
\label{eq:prb:ex1}
\begin{gathered}
\mathcal{L}\left( y\right) = \int_{a}^{b}
\sqrt{1+\left( \tilde{D}_{q,\omega } \left[ y\right]
\left( t\right) \right) ^{2}} \tilde{d}_{q,\omega}t
\longrightarrow \min \\
y\in \mathcal{Y}^{1}\left( \left[ a,b\right]_{q,\omega },\mathbb{R}\right)\\
y\left( a\right) =a,\quad y\left( b\right) =b.
\end{gathered}
\end{equation}
If $y_{\ast }$ is a local minimizer to the problem, then $y_{\ast}$
satisfies the Hahn symmetric Euler--Lagrange equation
\begin{equation}
\label{eq:el:ex1}
\tilde{D}_{q,\omega }\left[ \tau \mapsto \frac{\left(\tilde{D}_{q,\omega }\left[
y\right] \right)^{\sigma }\left( \tau\right)}{\sqrt{1+\left(\left(
\tilde{D}_{q,\omega }\left[ y\right] \right)^{\sigma }\left( \tau \right)\right) ^{2}}} \right]
\left( t\right) =0 \ \text{ for all }\ t \in \left[ a,b\right]_{q,\omega}.
\end{equation}
It is simple to check that function $y_{\ast }\left( t\right) =t$
is a solution to \eqref{eq:el:ex1} satisfying the given boundary conditions.
Since the Lagrangian is jointly convex in $(u,v)$, then we conclude from
Theorem~\ref{Suficiente} that function $y_{\ast }\left( t\right) =t$
is indeed a minimizer to problem \eqref{eq:prb:ex1}.
\end{ex}


\subsection{Leitmann's Direct Method}
\label{L}

Similarly to Malinowska and Torres \cite{Malinowska}, we show that Leitmann's
direct method \cite{Leitmann3} has also applications in the Hahn symmetric
variational calculus. Consider the variational functional integral
\begin{equation*}
\mathcal{\bar{L}}\left( \bar{y}\right)
= \int_{a}^{b}\bar{L}\left( t,\bar{y}^{\sigma }\left( t\right),
\tilde{D}_{q,\omega }\left[ \bar{y}\right]\left(t\right)\right)
\tilde{d}_{q,\omega }t.
\end{equation*}
As before, we assume that function $\bar{L}:I\times \mathbb{R} \times \mathbb{R}
\rightarrow \mathbb{R}$ satisfies the following hypotheses:
\begin{enumerate}
\item[($\overline{\text{H1}}$)] $\left( u,v\right)
\rightarrow \bar{L}\left( t,u,v\right) $ is a
$C^{1}\left( \mathbb{R}^{2},\mathbb{R}\right)$ function for any $t\in I$;

\item[($\overline{\text{H2}}$)] $t\rightarrow \bar{L}\left( t,\bar{y}^{\sigma }\left( t\right),
\tilde{D}_{q,\omega }\left[ \bar{y}\right] \left( t\right)\right)$ is continuous at
$\omega_{0}$ for any admissible function $\bar{y}$;

\item[($\overline{\text{H3}}$)] functions $t\rightarrow \partial_{i+2}\bar{L}\left(t,
\bar{y}^{\sigma }\left( t\right), \tilde{D}_{q,\omega }\left[\bar{y}\right] \left(
t\right) \right)$ belong to $\mathcal{Y}^{1}\left(\left[a,b\right]_{q,\omega},
\mathbb{R}\right)$ for all admissible $\bar{y}$, $i=0,1$.
\end{enumerate}

\begin{lemma}[Leitmann's fundamental lemma via Hahn's symmetric quantum operator]
Let $y=z\left( t,\bar{y}\right) $ be a transformation having a unique inverse
$\bar{y}=\bar{z}\left( t,y\right)$ for all $t\in \left[ a,b\right]_{q,\omega }$,
such that there is a one-to-one correspondence
\begin{equation*}
y\left( t\right) \leftrightarrow\bar{y}\left( t\right)
\end{equation*}
for all functions $y\in \mathcal{Y}^{1}\left(\left[a,b\right]_{q,\omega},\mathbb{R}\right)$
satisfying the boundary conditions $y\left( a\right) =\alpha $ and $y\left( b\right) =\beta$
and all functions $\bar{y}\in \mathcal{Y}^{1}\left(\left[a,b\right]_{q,\omega},
\mathbb{R}\right)$ satisfying
\begin{equation}
\label{leit}
\bar{y}\left(a\right)=\bar{z}\left( a,\alpha \right)
\text{ and }\bar{y}\left(b\right)=\bar{z}\left( b,\beta \right) \text{.}
\end{equation}
If the transformation $y=z\left( t,\bar{y}\right)$ is such that there exists
a function $G:I\times \mathbb{R} \rightarrow \mathbb{R}$ satisfying the identity
\begin{equation*}
L\left( t,y^{\sigma }\left( t\right) ,\tilde{D}_{q,\omega }\left[
y\right] \left( t\right) \right) -\bar{L}\left( t,\bar{y}^{\sigma }\left( t\right),
\tilde{D}_{q,\omega }\left[ \bar{y}\right] \left( t\right) \right)
=\tilde{D}_{q,\omega }\left[\tau \mapsto G\left(
\tau,\bar{y}\left( \tau\right)\right)\right](t) ,
\end{equation*}
$\forall t\in [a,b]_{q,\omega}$, then if $\bar{y}_{\ast}$ is a maximizer (resp. minimizer)
of $\mathcal{\bar{L}}$ with $\bar{y}_{\ast }$ satisfying \eqref{leit},
$y_{\ast }=z\left( t,\bar{y}_{\ast}\right)$ is a maximizer (resp. minimizer)
of $\mathcal{L}$ for $y_{\ast }$ satisfying $y_{\ast }\left( a\right) =\alpha$
and $y_{\ast }\left( b\right) =\beta$.
\end{lemma}

\begin{proof}
Suppose $y\in \mathcal{Y}^{1}\left( \left[ a,b\right]_{q,\omega },\mathbb{R}\right)$
satisfies the boundary conditions $y\left( a\right) =\alpha $ and $y\left(b\right) =\beta$.
Define function $\bar{y}\in \mathcal{Y}^{1}\left( \left[ a,b\right]_{q,\omega },\mathbb{R}\right)$
through the formula $\bar{y}=\bar{z}\left(t,y\right) $, $t\in \left[ a,b\right] _{q,\omega }$.
Then, $\bar{y}$ satisfies \eqref{leit} and
\begin{equation*}
\begin{split}
\mathcal{L}&\left( y\right) -\mathcal{\bar{L}}\left( \bar{y}\right)\\
&=\int_{a}^{b}L\left( t,y^{\sigma }\left( t\right), \tilde{D}_{q,\omega}\left[
y\right]\left( t\right) \right) \tilde{d}_{q,\omega }t
-\int_{a}^{b}\bar{L}\left( t,\bar{y}^{\sigma }\left( t\right),
\tilde{D}_{q,\omega }\left[ \bar{y}\right]
\left( t\right) \right) \tilde{d}_{q,\omega }t \\
&= \int_{a}^{b}\tilde{D}_{q,\omega }\left[\tau
\mapsto G\left(\tau, \bar{y}\left(\tau\right) \right)\right](t) \,
\tilde{d}_{q,\omega }t \\ &= G\left( b,\bar{y}\left( b\right) \right)
- G\left( a,\bar{y}\left(a\right) \right)\\ &= G\left( b,\bar{z}\left( b,\beta \right)\right)
- G\left( a,\bar{z}\left(a,\alpha \right) \right) .
\end{split}
\end{equation*}
The desired result follows immediately because the right-hand side of the above equality
is a constant, depending only on the fixed-endpoint conditions
$y\left(a\right) =\alpha $ and $y\left( b\right) =\beta$.
\end{proof}

\begin{ex}
Let $q\in\left]0,1\right[$, $\omega \geq 0$, and $a:=\omega_{0},b$
with $\omega_{0}<b$ be fixed real numbers. Also, let $I$ be an interval
of $\mathbb{R}$ such that $\omega _{0},b\in I$. We consider the problem
\begin{equation}
\label{Lproblem}
\begin{gathered}
\mathcal{L}\left( y\right)
=\int_{a}^{b}\left( \left( \tilde{D}_{q,\omega}\left[ y\right]\left( t\right) \right)^{2}
+ qy^{\sigma }\left( t\right) + t \tilde{D}_{q,\omega }\left[ y\right] \left( t\right) \right)
\tilde{d}_{q,\omega }t \longrightarrow \min \\ y\in \mathcal{Y}^{1}\left(
\left[ a,b\right]_{q,\omega },\mathbb{R}\right)\\ y\left( a\right) =\alpha,
\quad y\left( b\right) =\beta,
\end{gathered}
\end{equation}
where $\alpha ,\beta \in \mathbb{R}$ and $\alpha \neq \beta$.
We transform problem \eqref{Lproblem} into the trivial problem
\begin{equation*}
\begin{gathered}
\mathcal{\bar{L}}\left( \bar{y}\right)
= \int_{a}^{b} \left( \tilde{D}_{q,\omega }\left[ \bar{y}\right] \left( t\right) \right)^{2}
\tilde{d}_{q,\omega} t\longrightarrow \min \\
\bar{y}\in \mathcal{Y}^{1}\left( \left[ a,b\right] _{q,\omega },\mathbb{R}\right)\\
\bar{y}\left( a\right) =0, \quad \bar{y}\left(b\right) =0,
\end{gathered}
\end{equation*}
which has solution $\bar{y}\equiv 0$. For that we consider the transformation
\begin{equation*}
y\left( t\right) =\bar{y}\left( t\right) +ct+d,
\end{equation*}
where $c,d$ are real constants that will be chosen later.
Since $y^{\sigma }\left( t\right) =\bar{y}^{\sigma }\left( t\right)
+c\sigma \left(t\right) +d$ and $\tilde{D}_{q,\omega }\left[ y\right] \left( t\right)
=\tilde{D}_{q,\omega} \left[ \bar{y}\right] \left( t\right) +c$, we have
\begin{equation*}
\begin{split}
&\left( \tilde{D}_{q,\omega }\left[ y\right] \left( t\right) \right)^{2}
+ qy^{\sigma }\left( t\right) +t\tilde{D}_{q,\omega }\left[ y\right] \left( t\right) \\
&=\left( \tilde{D}_{q,\omega }\left[ \bar{y}\right] \left( t\right) \right)^{2}
+2c\tilde{D}_{q,\omega }\left[ \bar{y}\right] \left( t\right) + c^{2}+q d
+ q\bar{y}^{\sigma}\left(t\right) + t\tilde{D}_{q,\omega }\left[
\bar{y}\right] \left( t\right) + c\left( q\sigma \left( t\right) +t\right).
\end{split}
\end{equation*}
Therefore,
\begin{equation*}
\begin{split}
&\left[ \left( \tilde{D}_{q,\omega }\left[ y\right] \left( t\right) \right)^{2}
+qy^{\sigma }\left( t\right) +t\tilde{D}_{q,\omega }\left[ y\right] \left(t\right) \right]
-\left( \tilde{D}_{q,\omega }\left[ \bar{y}\right] \left( t\right) \right)^{2}\\
&= \tilde{D}_{q,\omega }\left[ 2c\bar{y}\right] \left( t\right)
+\tilde{D}_{q,\omega }\left[ \left( c^{2}+qd\right) id\right] \left( t\right)
+ \tilde{D}_{q,\omega }\left[ \sigma \cdot \bar{y}\right] \left( t\right)
+ c \tilde{D}_{q,\omega }\left[ \sigma \cdot id\right] \left( t\right) \\
&= \tilde{D}_{q,\omega }\left[ 2c\bar{y}+\left( c^{2}+qd\right) id
+\sigma \cdot \bar{y}+c\left( \sigma \cdot id\right) \right] \left( t\right),
\end{split}
\end{equation*}
where $id$ represents the identity function. In order to obtain the solution
to the original problem, it suffices to choose $c$ and $d$ such that
\begin{equation}
\label{sol}
\left\{
\begin{array}{c}
ca+d=\alpha \\
cb+d=\beta .
\end{array}
\right.
\end{equation}
Solving the system of equations \eqref{sol} we obtain
$c=\displaystyle\frac{\alpha -\beta}{a-b}$ and
$d=\displaystyle\frac{a\beta -b\alpha }{a-b}$.
Hence, the global minimizer to problem \eqref{Lproblem} is
\begin{equation*}
y\left( t\right) =\frac{\alpha -\beta }{a-b}t
+\frac{a\beta -b\alpha }{a-b}.
\end{equation*}
\end{ex}


\section*{Acknowledgments}

Work supported by {\it FEDER} funds through {\it COMPETE} --- Operational
Programme Factors of Competitiveness (``Programa Operacional Factores de
Competitividade'') and by Portuguese funds through the {\it Center
for Research and Development in Mathematics and Applications} (University of Aveiro)
and the Portuguese Foundation for Science and Technology
(``FCT --- Funda\c{c}\~{a}o para a Ci\^{e}ncia e a Tecnologia''),
within project PEst-C/MAT/UI4106/2011 with COMPETE number
FCOMP-01-0124-FEDER-022690. The first author was also supported
by FCT through the Ph.D. fellowship SFRH/BD/33634/2009.



\medskip
Received December 2011; revised September 2012.
\medskip



\begin{thebibliography}{99}

\bibitem{Aldwoah}
K. A. Aldwoah,
\emph{Generalized time scales and associated difference equations},
Ph.D. thesis, Cairo University, 2009.

\bibitem{Aldwoah2} (MR2918250)
K. A. Aldwoah, A. B. Malinowska\ and\ D. F. M. Torres,
\emph{The power quantum calculus and variational problems},
Dyn. Contin. Discrete Impuls. Syst. Ser. B Appl. Algorithms
{\bf 19} (2012), no.~1-2, 93--116.
{\tt arXiv:1107.0344}

\bibitem{MyID:154} (MR2727133)
R. Almeida\ and\ D. F. M. Torres,
\emph{Leitmann's direct method for fractional optimization problems},
Appl. Math. Comput. {\bf 217} (2010), no.~3, 956--962.
{\tt arXiv:1003.3088}

\bibitem{MyID:188} (MR2861739)
R. Almeida\ and\ D. F. M. Torres,
\emph{Nondifferentiable variational principles in terms of a quantum operator},
Math. Methods Appl. Sci. {\bf 34} (2011), no.~18, 2231--2241.
{\tt arXiv:1106.3831}

\bibitem{Boole} (MR0115025)
G. Boole,
{\it Calculus of finite differences},
Edited by J. F. Moulton 4th ed, Chelsea Publishing Co. New York, 1957.

\bibitem{Cruz} (MR2861326)
A. M. C. Brito da Cruz, N. Martins\ and\ D. F. M. Torres,
\emph{Higher-order Hahn's quantum variational calculus},
Nonlinear Anal. 75 (2012), no.~3, 1147--1157.
{\tt arXiv:1101.3653}

\bibitem{Cruz2}
A. M. C. Brito da Cruz\ and\ N. Martins,
The $q$-symmetric variational calculus,
Comput. Math. Appl., in press.
DOI:10.1016/j.camwa.2012.01.076

\bibitem{Leitmann1} (MR2187205)
D. A. Carlson\ and\ G. Leitmann,
\emph{Coordinate transformation method for the extremization of multiple integrals},
J. Optim. Theory Appl. {\bf 127} (2005), no.~3, 523--533.

\bibitem{Leitmann2} (MR2373538)
D. A. Carlson\ and\ G. Leitmann,
\emph{Fields of extremals and sufficient conditions
for the simplest problem of the calculus of variations},
J. Global Optim. {\bf 40} (2008), no.~1-3, 41--50.

\bibitem{MyID:111} (MR2549615)
J. Cresson, G. S. F. Frederico\ and\ D. F. M. Torres,
\emph{Constants of motion for non-differentiable quantum variational problems},
Topol. Methods Nonlinear Anal. {\bf 33} (2009), no.~2, 217--231.
{\tt arXiv:0805.0720}

\bibitem{Ernst} (MR2554406)
T. Ernst,
\emph{The different tongues of $q$-calculus},
Proc. Est. Acad. Sci. {\bf 57} (2008), no.~2, 81--99.

\bibitem{Feynman} (MR2797644)
R. P. Feynman\ and\ A. R. Hibbs,
{\it Quantum mechanics and path integrals},
emended edition, Dover, Mineola, NY, 2010.

\bibitem{Hahn} (MR0030647)
W. Hahn,
\emph{\"Uber Orthogonalpolynome,
die $q$-Differenzengleichungen gen\"ugen},
Math. Nachr. {\bf 2} (1949), 4--34.

\bibitem{Jackson} (MR1506108)
F. H. Jackson,
\emph{$q$-Difference equations},
Amer. J. Math. {\bf 32} (1910), no.~4, 305--314.

\bibitem{Kac} (MR1865777)
V. Kac\ and\ P. Cheung,
{\it Quantum calculus},
Universitext, Springer, New York, 2002.

\bibitem{Koekoek} (MR2656096)
R. Koekoek, P. A. Lesky\ and\ R. F. Swarttouw,
{\it Hypergeometric orthogonal polynomials and their $q$-analogues},
Springer Monographs in Mathematics, Springer, Berlin, 2010.

\bibitem{Lavagno2}
A. Lavagno\ and\ G. Gervino,
\emph{Quantum mechanics in $q$-deformed calculus},
J. Phys.: Conf. Ser. {\bf 174} (2009), 012071, 8~pp.

\bibitem{Lavagno} (MR1923946)
A. Lavagno\ and\ P. Narayana Swamy,
\emph{$q$-deformed structures and nonextensive statistics: a comparative study},
Phys. A {\bf 305} (2002), no.~1-2, 310--315.

\bibitem{Leitmann3} (MR0230187)
G. Leitmann,
\emph{A note on absolute extrema of certain integrals},
Internat. J. Non-Linear Mech. {\bf 2} (1967), 55--59.

\bibitem{Leitmann4} (MR1828668)
G. Leitmann,
\emph{On a class of direct optimization problems},
J. Optim. Theory Appl. {\bf 108} (2001), no.~3, 467--481.

\bibitem{Leitmann5} (MR1850676)
G. Leitmann,
\emph{Some extensions to a direct optimization method},
J. Optim. Theory Appl. {\bf 111} (2001), no.~1, 1--6.

\bibitem{Leitmann6} (MR1954118)
G. Leitmann,
\emph{On a method of direct optimization},
Vychisl. Tekhnol. {\bf 7} (2002), 63--67.

\bibitem{Malinowska2} (MR2727155)
A. B. Malinowska\ and\ D. F. M. Torres,
\emph{Leitmann's direct method of optimization for absolute extrema
of certain problems of the calculus of variations on time scales},
Appl. Math. Comput. {\bf 217} (2010), no.~3, 1158--1162.
{\tt arXiv:1001.1455}

\bibitem{Malinowska} (MR2733985)
A. B. Malinowska\ and\ D. F. M. Torres,
\emph{The Hahn quantum variational calculus},
J. Optim. Theory Appl. {\bf 147} (2010), no.~3, 419--442.
{\tt arXiv:1006.3765}

\bibitem{MyID:146} (MR2793813)
N. Martins\ and\ D. F. M. Torres,
\emph{L'H\^{o}pital-type rules for monotonicity
with application to quantum calculus},
Int. J. Math. Comput. {\bf 10} (2011), M11, 99--106.
{\tt arXiv:1011.4880}

\bibitem{MyID:220}
N. Martins\ and\ D. F. M. Torres,
\emph{Higher-order infinite horizon variational problems
in discrete quantum calculus},
Comput. Math. Appl., in press.
DOI:10.1016/j.camwa.2011.12.006
{\tt arXiv:1112.0787}

\bibitem{Strominger} (MR1247149)
D. N. Page,
\emph{Information in black hole radiation},
Phys. Rev. Lett. {\bf 71} (1993), no.~23, 3743--3746.

\bibitem{delfim:GL} (MR2386612)
D. F. M. Torres\ and\ G. Leitmann,
\emph{Contrasting two transformation-based methods for obtaining absolute extrema},
J. Optim. Theory Appl. {\bf 137} (2008), no.~1, 53--59.
{\tt arXiv:0704.0473}

\bibitem{Youm}
D. Youm,
\emph{$q$-deformed conformal quantum mechanics},
Phys. Rev. D {\bf 62} (2000), 095009, 5~pp.

\end{thebibliography}
\end{document}